\documentclass{article}

\usepackage{arxiv}

\usepackage[utf8]{inputenc} % allow utf-8 input
\usepackage[T1]{fontenc}    % use 8-bit T1 fonts
\usepackage{hyperref}       % hyperlinks
\usepackage{url}            % simple URL typesetting
\usepackage{booktabs}       % professional-quality tables
\usepackage{amsfonts}       % blackboard math symbols
\usepackage{nicefrac}       % compact symbols for 1/2, etc.
\usepackage{microtype}      % microtypography
\usepackage{wrapfig}
\usepackage{subcaption}
\usepackage[export]{adjustbox}
\usepackage{graphicx}
\usepackage{amsmath}
\usepackage{amsthm}
\usepackage{cancel}
\usepackage{float}
\newtheorem{thm}{Theorem}[section]

\newtheorem{lem}[thm]{Lemma}
\newtheorem{prop}[thm]{Proposition}
\newtheorem{cor}[thm]{Corollary}
\theoremstyle{definition}
\newtheorem{exmp}{Example}[section]
\theoremstyle{remark}

\theoremstyle{definition}
\newtheorem{defn}{Definition}[section]

\theoremstyle{plain}

\title{On the impossibility of certain $({n^2+n+k}_{n+1})$ configurations}

\author{
  Jackson Philbrook, Benjamin Peet\\
  Department of Mathematics\\
  St. Martin's University\\
  Lacey, WA 98503 \\
  \texttt{bpeet@stmartin.edu} \\
}

\begin{document}
\maketitle

\begin{abstract}
This paper investigates the impossibility of certain $({n^2+n+k}_{n+1})$ configurations. Firstly, for $k=2$, the result of \cite{gropp1992non} that $\frac{n^2+n}{2}$ is even and $n+1$ is a perfect square or $\frac{n^2+n}{2}$ is odd and $n-1$ is a perfect square is reproved using the incidence matrix $N$ and analysing the form of $N^TN$. Then, for all $k$, configurations where paralellism is a transitive property are considered. It is then analogously established that if $n\equiv0$ or $n\equiv k-1$ mod $k$ for $k$ even, then $\frac{n^2+n}{k}$ is even and $n+1$ is a perfect square or $\frac{n^2+n}{k}$ is odd and $n-(k-1)$ is a perfect square. Finally, the case $k=3$ is investigated in full generality.
\end{abstract}

% keywords can be removed

\section*{Acknowledgement}

We would like to acknowledge the contribution of Hayden Jones and Michael Rivet to this paper.

\section{Introduction}

 It has been established that there are no order 6 and order 10 projective planes and Theorem 1 from \cite{bruck_ryser_1949} establishes that if $n$ is congruent to 1 or 2 modulo 4, then $n$ must be a sum of two squares. This establishes that certain projective planes cannot exist.

Projective planes are particular examples of symmetric $(n^2+n+1_{n+1})$ configurations of points and lines. We define a configuration as in\cite{grunbaum2009configurations} as consisting of $v$ "points" and $r$ "lines". Each of the points lies on exactly $b$ lines, and each of the lines lies on exactly $k$ points. Two lines lie on at most one shared point. The terms "point" and "line" could be substituted as long as their incidence relation (where an incidence involves exactly one point and exactly one line) remains true. 

Combinatorically, a configuration is given by a pair $(\mathcal{P},\mathcal{L})$ where $\mathcal{P}=\{p_{1}, \ldots , p_{v}\} $; $\mathcal{L}=\{l_{1}, \ldots ,l_{b}\}$; each $l_{i}\subset \mathcal{P}$; and for any pair $p_{i_{1}},p_{i_{2}}$ there is at most one line $l_{j}$ that contains both elements; each point is incident with the same number of lines as any other (denoted $b$); each line is incident with the same number of points as any other (denoted $r$); and there are at least three points. We make the assumption that the space is connected in the sense that the collection of points cannot be split without splitting a line.

We denote such a configuration as a $(v_b,r_k)$ configuration, or $(v_b)$ symmetric configurations when $v=r$ and consequently $b=k$.

This paper looks to investigate the impossibility of certain $({n^2+n+k}_{n+1})$ symmetric configurations that have deficiency $k$ - informally how far they are from being a projective plane.

We first take the case $k=2$ and restate Corollary 2.3 of \cite{gropp1992non} with our own proof. This proof relies upon the incidence matrix $N$ and the particular form of $N^{T}N$ where $(N^{T}N)_{ij}$ describes how many times a line $l_i$ intersects a line $l_j$.

We then continue by considering configurations where parallelism is a transitive property according to the following definition:

\newtheorem*{def:1}{Definition \ref{def:1}}
\begin{def:1}

    We say that a configuration has the transitivity of parallels property if $l_1$ is parallel to both $l_2$ and $l_3$, then necessarily $l_2$ is parallel to $l_3$.
\end{def:1}

We then establish the following result for the impossibility of such configurations:

\newtheorem*{thm:1}{Theorem \ref{thm:1}}
\begin{thm:1}
Suppose that $n\equiv0$ or $n\equiv k-1$ mod $k$ for $k$ even. If a $(n^{2}+n+k_{n+1)}$ configuration satisfies the TOPs property then:
    
1. $\frac{n^2+n}{k}$ is even and $n+1$ is a perfect square

or:

2. $\frac{n^2+n}{k}$ is odd and $n-(k-1)$ is a perfect square. 

\end{thm:1}

We will show that these configurations are a generalization of $k$-nets \cite {bogya2015classification}, which are themselves generalizations of affine planes.

Finally, we pick $k=3$ and investigate the possible forms of $N^{T}N$ to yield the result:

\newtheorem*{thm:2}{Theorem \ref{thm:2}}
\begin{thm:2}
    If the lines of a configuration with deficiency $k=3$ are partitioned into $m$ cycles of lines where consecutive lines are non-intersecting, then either the configuration is configuration 9.1, or $m$ is odd.
\end{thm:2}

We make repeated use of the following preliminary definitions:

We say two lines $l_1$ and $l_2$ are non intersecting if the sets of points that make up each line share no elements. That is, that there exists no point that lies on both $l_1$ and $l_2$. \cite{dembowski1997finite} We interchangably use the terms parallel and non-intersecting.

A balanced incomplete block design (BIBD) is a similar concept to a configuration, but includes the extra parameter $\lambda$ and is defined  by consisting of $v$ points and $r$ lines. Each of the points lies on exactly $b$ lines, and each of the lines lies on exactly $k$ points. Finally, any two lines lie on exactly $\lambda$ points.\cite{dembowski1997finite}

A symmetric configuration (or BIBD) is such that $v=b$ and as $vr=bk$ \cite{grunbaum2009configurations}, $r=k$. We denote such symmetric configurations as $(v_r)$ configurations.

When $\lambda=1$, we have any two lines lie on one exactly one point, that is, all lines intersect. These are the finite projective planes and is the intersection of BIBDs and configurations.

Necessarily, a finite projective plane is symmetric and must have $n^2+n+1$ points (and lines) and $n+1$ points per line (and lines per point) for some natural number $n$. \cite{dembowski1997finite} Hence, $(n^2+n+1_{n+1})$ configurations where $n$ is termed the order of the projective plane.

A symmetric configuration is cyclic if its' point set can be denoted as $\{0,\ldots,v-1\}$ and then its' lines are given by $\{p_1+j \text{ mod } v, \ldots, p_k +j\text{ mod } v\}$ for $j=0,\ldots,v-1$ and some distinct $p_1,\dots,p_k \in \{0,\ldots,v-1\}$.

Given, a $(n^2+n+k_{n+1})$ configuration, $k$ is termed the deficiency and is in some sense how far from a projective plane a symmetric configuration is. The parameter $k$ will only be used in this context from here forward (not as the number of lines per point).

For a configuration (or BIBD) the incidence matrix is the binary matrix $N$ of dimension $v\times b$ such that:

$$n_{ij}=\begin{cases} 
      0 & \text{point $p_i$ is not incident with $l_j$} \\
      1 & \text{point $p_i$ is incident with $l_j$} \\
   \end{cases}
$$

A circulant matrix $A$ with first row $c_0,c_1,\ldots, c_{t-1}$ is a square matrix of size $t \times t$, where each consecutive row is the previous row shifted one column to the right, that is, each element $a_{ij}=a_{(i-1)((j-1) \text{ modulo } t)}$ for $i=1,\ldots,t-1$. The matrix then looks as follows:

$$A= \begin{bmatrix}
c_0 & c_1 &\ldots & c_{t-1} \\
c_{t-1} & c_0 & \ldots & c_{t-2} & \\
\vdots & \vdots &&\vdots\\
c_1&c_2&\ldots&c_0\\

\end{bmatrix}$$

These circulant matrices will be a useful tool in our final Theorem

\section{Preliminary results}

We repeatedly use the proof of Theorem 1 of \cite{shrikhande1950impossibility} as a model for proving our results:

\begin{thm}
    A necessary condition for the existence of a symmetric balanced incomplete block design with parameters $v,r$ and $\lambda$ where $v$ is even is that $r-\lambda$ is a perfect square.
\end{thm}

We note again that in the case $\lambda=1$, we do reduce to a restricted form of a combinatorial configuration, a projective plane. However, given a projective plane, the number of points is necessarily odd. Hence this particular theorem does not work for our purposes. The proof however, will form the template for proving our results.

\subsection{Examples}

Theorem 1 from \cite{bruck_ryser_1949} (in adapted form) states that:

\begin{thm}
    If there exists a projective plane of finite order $n$ and if $n$ is congruent to $1$ or $2$ modulo $4$, then $n$ must be a sum of two squares.
\end{thm}

This proves that there is no combinatorial configuration $({n^2+n+1}_{n+1})$ for $n=6$ in particular.

We now consider some known examples with $n=2$ and note that there are indeed $({n^2+n+k}_{n+1})=({6+k}_{3})$ configurations for all $k$. The first few examples are the $(8_3)$ configuration of Mobius-Kantor \cite{coxeter1950self}; the $(9_3)$ configuration of Pappus \cite{grunbaum2009configurations}; and the $(10_3)$ configuration of Desargues \cite{pisanski2012configurations}.

Indeed, \cite{betten2000counting} enumerates the number of such symmetric configurations.

\section{$k=2$}

\subsection{Theorem 3.1}

We state our first result which is Corollary 2.3 of \cite{gropp1992non}. In that paper, the proof given uses Theorem 9 of \cite{bose1952combinatorial}. We give our own proof that will be generalized for greater deficiencies in section 5.

\begin{thm}

There exists a $(n^{2}+n+2_{n+1})$ configuration only if either:

1. $\frac{n^2+n}{2}$ is even and $n+1$ is a perfect square

or:

2. $\frac{n^2+n}{2}$ is odd and $n-1$ is a perfect square.

\end{thm}

\begin{proof}

We mimic the proof of Theorem 1 from \cite{shrikhande1950impossibility}. Note that $n^2+n+2$ is always even and that each line has exactly one line that it does not intersect.

Without loss of generality we pair off the non-intersecting lines as $(1,2),\ldots,(n^2+n+1,n^2+n+2)$.

Now, the matrix $B=N^TN$ where $N$ is the incidence matrix is such that $b_{ij}$ represents how many times line $l_i$ intersects line $l_j$. This will be either $0$ or $1$ if $i\neq j$ and will be $n+1$ if $i=j$.

In particular, it will have form:

$$\begin{bmatrix}
n+1 & 0 & 1 & 1 & \ldots &1 & 1\\
0 & n+1 & 1 & 1 & \ldots &1 & 1 \\
1 & 1 & n+1 & 0 &  & 1 & 1 \\
1 & 1 & 0 & n+1 &  & 1 & 1 \\
\vdots & \vdots & & & \ddots&   & \\
1 & 1 &1&1&& n+1 & 0 \\
1 & 1 &1&1&& 0 & n+1 \\

\end{bmatrix}$$

That is, an $n^2+n+2$ by $n^2+n+2$ matrix of ones with $2$ by $2$ matrices on the diagonal of the form: $\begin{bmatrix} n+1 & 0 \\ 0 & n+1 \end{bmatrix}$.

We compute the determinant of $B$. Mimicing \cite{shrikhande1950impossibility}, we subtract column 1 from all other columns and then add all rows to row 1. This yields the following matrix form:

$$\begin{bmatrix}
n+1+n^2+n & 0 & 0 & 0 & \ldots &0 & 0\\
0 & n+1 & 1 & 1 & \ldots &1 & 1 \\
1 & 0 & n & -1 &  & 0 & 0 \\
1 & 0 & -1 & n &  & 0 & 0 \\
\vdots & \vdots & & & \ddots&   & \\
1 & 0 &0&0&& n & -1 \\
1 & 0 &0&0&& -1 & n \\

\end{bmatrix}$$

Noting that $n^2+2n+1=(n+1)^2$ and computing the determinant by expanding along the first row yields:

$$(n+1)^2\begin{vmatrix}
n+1 & 1 & 1 & \ldots &1 & 1 \\
0 & n & -1 &  & 0 & 0 \\
0 & -1 & n &  & 0 & 0 \\
\vdots & & & \ddots&   & \\
0 &0&0&& n & -1 \\
0 &0&0&& -1 & n \\

\end{vmatrix}$$

Then expanding along the first column yields:

$$(n+1)^3\begin{vmatrix}

n & -1 &  & 0 &0  \\
-1 & n &  & 0 & 0\\
 & & \ddots &   & \\
0&0&& n & -1 \\
0&0&& -1 & n \\

\end{vmatrix}$$

This is then:

$$(n+1)^3\begin{vmatrix} n & -1 \\ -1 & n \end{vmatrix}^\frac{n^2+n}{2}=(n+1)^3(n^2-1)^\frac{n^2+n}{2}=(n+1)^{3+\frac{n^2+n}{2}}(n-1)^\frac{n^2+n}{2}$$

Given that $|B|=|NN^T|=|N|^2$ is a square number, we must have that either:

1. $\frac{n^2+n}{2}$ is even and $n+1$ is a perfect square

or:

2. $\frac{n^2+n}{2}$ is odd and $n-1$ is a perfect square.

\end{proof}

\subsection{Examples}
Using this formula, the first
five values of $n$ that do not satisfy these requirements are $n=4, 6, 7 ,9,11$. That is, there are no combinatorial configurations with parameters $(22_5),(44_7),(58_8),(92_{10})$ and $(134_{12})$.

\section{Configurations with the TOPs property}

\subsection{Definitions and preliminary results}
We now consider the first $15_4$ configuration from \cite{betten1999tactical}. Note that the deficiency $k=3$. Therefore, each line has two non-intersecting lines. In this example, given a line, the two non-intersecting lines themselves intersect.

In the third $15_4$ configuration from \cite{betten1999tactical} it is true that if a line is non-intersecting with two lines, then those two lines are also non-intersecting.

Given that the matrix $B=N^TN$ will be significantly easier to determine for the second case, we make our definition of the transitivity of parallels property and proceed to restrict our attention to them.

\begin{defn}
\label{def:1}
    We say that a configuration has the transitivity of parallels property if $l_1$ is parallel to both $l_2$ and $l_3$, then necessarily $l_2$ is parallel to $l_3$.
\end{defn}

Playfair's axiom \cite{playfair1849elements} in the context of Euclidean (non-finite) geometry states that "In a plane, given a line and a point not on it, at most one line parallel to the given line can be drawn through the point."

In some sources it is stated as "exactly one line parallel" as in Euclidean plane geometry this follows from the other axioms. However, we deliberately choose the original definition and call it Playfair's condition for configurations:

"In a configuration, given a line and a point not on it, there is at most one line parallel to the given line incident to the given point."

\begin{prop}
    The TOPs property is equivalent to Playfair's condition for configurations.
\end{prop}

\begin{proof}
    First we show that TOPs implies Playfair's. For contradiction let Playfair's be false, then given a line $l_1$, and a point $p$ not on $l_1$ there are at least two lines $l_2,l_3$ incident with $p$ both parallel to $l_1$. Hence $l_2$ and $l_3$ are not parallel, contradicting the TOPs property.

    Similarly, to show that Playfair's implies TOPs, let Playfair's hold true. For the sake of contradiction, suppose the TOPs property is false and that for $l_1$ parallel to both $l_2$ and $l_3$, $l_2$ and $l_3$ are not parallel. Hence they meet at a shared point $p$. There are now two lines on point $p$ that are parallel to $l_1$, which contradicts Playfair's. 
\end{proof}

\begin{prop}
    If a $({n^2+n+k}_{n+1})$ configuration satisfies the TOPs property (equivantly Playfair's condition), then:
    $$k\leq n+1$$
\end{prop}

\begin{proof}
    Given a set of $k$ mutually non-intersecting lines, we have at a minimum $k(n+1)$ points. Hence $n^2+n+k\geq k(n+1)$ and consequently $n+1\geq k$. 
\end{proof}

\subsection{Examples and connections to affine planes and $k$-nets}
We first give an example of a configuration with the TOPs property:

\begin{exmp}
    The Pappus configuration, with $n=2$ and $k=3$.

    \begin{figure}[ht]
\centering
\includegraphics[height=6cm]{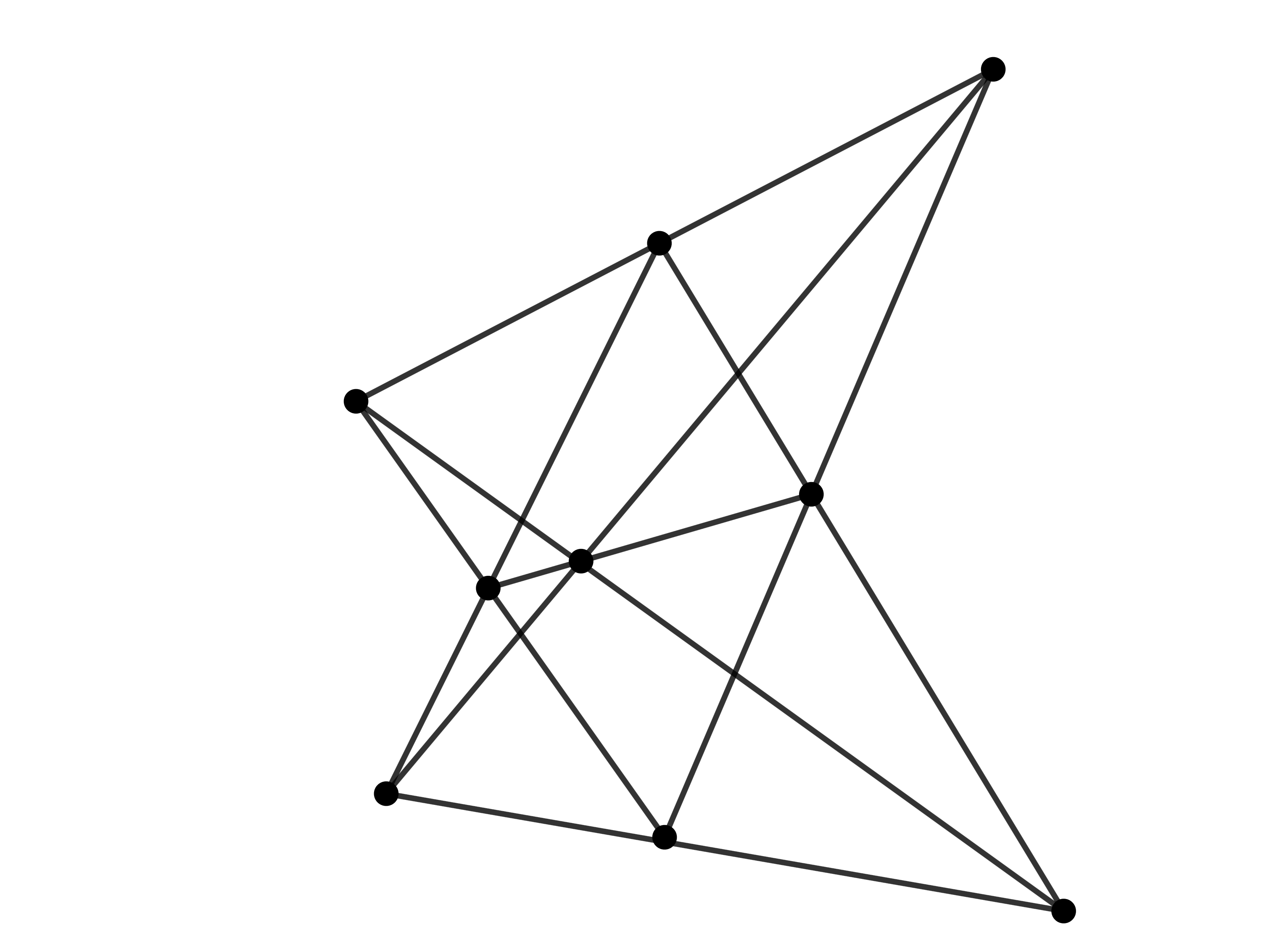}
\caption{Pappus' configuration}
\end{figure}

Note the three sets of mutually parallel lines.
\end{exmp}

We make here the observation that parallelism is also transitive for finite affine planes. The difference between these types of configurations and affine planes is that for an affine plane, any two points must both be incident with some line. This is not true for configurations with the TOPs property. For the Pappus configuration above, it can be seen that the two left most points do not lie on a line together. Hence, while the Pappus configuration has the TOPs property, it is not affine. Indeed, no symmetric configuration can be affine. \cite{grunbaum2009configurations}

A generalization of an affine plane is a $k$-net \cite{bogya2015classification}, in this case parallelism is again transitive, each point must lie on a line from each of the parallel classes and there are $k$ parallel classes. In this scenario, the Pappus configuration is a $3$-net of order $3$. However, for $k$-nets, any set of parallel lines partitions the points. This is not true for configurations with the TOPs property in general. We offer up the example of the third $15_4$ configuration from \cite{betten1999tactical}. Given each line has $4$ points, it is not possible that a set of mutually parallel lines could partition the $15$ points.

We formalize this as the following:

\begin{prop}
A $(n^2+n+k_{n+1})$ configuration with the TOPs property is a $k$-net if and only if $k=n+1$.
\end{prop}

\begin{proof}
    If $k=n+1$ then each line does not intersect with $n$ other lines. Given the TOPs property, these lines are each mutually non-intersecting. Hence there are at minimum $n+1+n(n+1)=n^2+n+(n+1)$ points. This is all the points and so the set of mutually parallel lines partition the points. 

    For the converse, if a set of mutually parallel lines partition the set of $n^2+n+k$ points, then $n^2+n+k$ is divisible by $n+1$. Hence $n+1$ divides $k$. However, by Proposition 4.2, $k\leq n+1$. Hence $k=n+1$.
\end{proof}

\subsection{Properties of configurations with the TOPs property}

We here address how having the TOPs property relates to various basic concepts of configurations.

Firstly, a TOPs configuration may or may not be cyclic, just as a cyclic configuration may or may not have the TOPs property. As examples; the Pappus configuration has the TOPs property but is not cyclic; the first $(15_4)$ configuration of \cite{betten1999tactical} is cyclic but does not have the TOPs property; and the third $(15_4)$ configuration is cyclic and has the TOPs property.

Secondly, we show that if a configuration has the TOPs property, then so does its' dual:

\begin{prop}
    Given an $(n^2+n+k_{n+1})$ configuration with the TOPs property, the dual $(n^2+n+k_{n+1})$ configuration also has the TOPs property.
\end{prop}

\begin{proof}
    First take a point $p_1$ and a line $l_1$ through $p_1$.
    
    There then exist $l_2,\ldots,l_k$ so that $l_1,l_2,\ldots,l_k$ are mutually parallel.  Necessarily, there are $n$ other lines through $l_1$, call these $l_1^1,l_1^2,\ldots,l_1^n$. Each of $l_2,\ldots,l_k$ have $n+1$ points and must necessarily have a point that does not meet any of $l_1^1,\ldots,l_1^n$. Call these points $p_2,\ldots,p_k$. Note that these $k-1$ points are the only points that $p_1$ does not share a line with as the sum of points of lines through $p_1$ is $n+1+n\cdot n=n^2+n+k-(k-1)$.

    Now, suppose for contradiction that given two points from the collection $p_2,\ldots,p_k$ there is a line incident with both. Without loss of generality, take $p_2$ and $p_3$ and call the line $l$.

    So now $l$ has $n+1$ points and already intersects $l_2$ and $l_3$, so can intersect at most $n-1$ of $l_1,l_1^1,\ldots,l_1^n$. Hence there are at least two that it does not intersect. Without loss of generality, we let these be $l_1$ and $l_1^1$. Hence $l$ is parallel to both $l_1$ and $l_1^1$, but $l_1$ and $l_1^1$ are not parallel - intersect at $p_1$. This contradicts the configuration having the TOPs property.

    Hence given $p_1$ and two points that it does not share a line with, those two points also do not share a line.

    Thus, for the dual configuration, given a line and two lines that it does not share a point with, those two lines also do not share a point. That is the dual configuration also has the TOPs property.
\end{proof}

\section{$k$ general with the TOPs property}
\subsection{Main results}
We begin with Lemma 5.1 and then an immediate corollary in the case where $k$ is even.
\begin{lem}
If a $(n^{2}+n+k_{n+1})$ configuration satisfies the TOPs property then $n^{2}+n$ is a multiple of $k$.
\end{lem}

\begin{proof}
    Dividing the lines into groups of $k$ mutually parallel lines yields $\frac{n^2+n+k}{k}=\frac{n^2+n}{k}+1$ groups. Hence, $\frac{n^2+n}{k}$ is a whole number and so $n^{2}+n$ is a multiple of $k$.
\end{proof}

\begin{cor}
    If a $(n^{2}+n+k_{n+1})$ configuration satisfies the TOPs property and $k$ is even then $n\equiv0$ or $n\equiv k-1$ mod $k$.
\end{cor}

\begin{proof}
    By Lemma 5.1, $n^2+n=mk$ for some natural number $m$. So $n(n+1)=mk$.
If $n$ is odd then $k$ must divide $n+1$, if $n$ is even then $n+1$ is odd and $k$ must divide $n$.
    Hence $n\equiv0$ or $n+1\equiv 0$ mod $k$. That is, $n\equiv0$ or $n\equiv k-1$ mod $k$

\end{proof}

We now can give our main result:

\begin{thm}
\label{thm:1}
    Suppose that $n\equiv0$ or $n\equiv k-1$ mod $k$ for $k$ even. If a $(n^{2}+n+k_{n+1})$ configuration satisfies the TOPs property then:

1. $\frac{n^2+n}{k}$ is even and $n+1$ is a perfect square

or:

2. $\frac{n^2+n}{k}$ is odd and $n-(k-1)$ is a perfect square.
    
\end{thm}

\begin{proof}
    We once again mimic the proof of Theorem 1 from \cite{shrikhande1950impossibility}.

Note that $n^2+n+k$ is always even for $k$ even. Given the result has been proven for $k=2$ in Theorem 3.1, we can assume $k\geq4$.

Note that each line has exactly $k-1$ lines that it does not intersect and due to the TOPs property we can without loss of generality relabel the lines so the $k$-tuples of mutually non-intersecting lines are $(1,\ldots,k),\ldots,(n^2+n+1,\dots,n^2+n+k)$.

Hence the matrix $B=N^TN$ will have the form:

$$B=\begin{bmatrix}
A & 1_{{k\times k}} &\ldots & 1_{{k\times k}} &1_{{k\times k}}\\
1_{{k\times k}} & A & & 1_{{k\times k}} & 1_{{k\times k}}\\
\vdots &&\ddots&&\vdots\\
1_{{k\times k}}&1_{{k\times k}}&&A&1_{{k\times k}}\\
1_{{k\times k}}&1_{{k\times k}}&\ldots &1_{{k\times k}}&A\\

\end{bmatrix}$$

Where $1_{{k\times k}}$ is a $k$ by $k$ matrix of ones and $A=(n+1)I_{{k\times k}}$.

We compute the determinant of $B$ by again subtracting column 1 from all other columns and then add all rows to row 1. This yields the following matrix form:

$$\begin{bmatrix}
n^2+2n+1 & 0_{1\times k-1} & 0_{1 \times k} &\ldots & 0_{1\times k} &0_{1\times k}\\
0_{k\times 1} & (n+1)I_{k-1\times k-1} & 1_{k-1\times k} & \ldots & 1_{k-1\times k} &1_{k-1\times k}\\
1_{k\times 1} &0_{k\times k-1} & A-1_k & & 0_k & 0_k\\
\vdots &\vdots&&\ddots&&\vdots\\
1_{k\times 1}&0_{k\times k-1}&0_k&&A-1_k&0_k\\
1_{k\times 1}&0_{k\times k-1}&0_k&\ldots &0_k&A-1_k\\

\end{bmatrix}$$

Computing the determinant by expanding along the first row yields:

$$(n+1)^2\begin{vmatrix}
 (n+1)I_{k-1\times k-1} & 1_{k-1\times k} & \ldots & 1_{k-1\times k} &1_{k-1\times k}\\
0_{k\times k-1} & A-1_k & & 0_k & 0_k\\
\vdots&&\ddots&&\vdots\\
0_{k\times k-1}&0_k&&A-1_k&0_k\\
0_{k\times k-1}&0_k&\ldots &0_k&A-1_k\\

\end{vmatrix}$$

Then as an upper triangular matrix, the determinant is 

$$(n+1)^2(n+1)^{k-1}|A-1_k|^\frac{n^2+n}{k}$$

By a similar argument of subtracting column 1 and adding all rows to row 1, we yield that:

$$|A-1_k|=\begin{vmatrix}
n-(k-1) & 0 & 0 & \ldots &0 & 0\\
-1  & n+1 & 0 &  & 0 & 0 \\
-1  & 0 & n+1 &  & 0 & 0 \\
\vdots  & & & \ddots&   & \\
-1  &0&0&& n+1 & 0 \\
-1  &0&0&& 0 & n+1 \\ \end{vmatrix}=(n-(k-1))(n+1)^{k-1}$$

So finally,

\begin{equation} \label{eq:1}
|B|=(n+1)^{k+1+(k-1)\frac{n^2+n}{k}}(n-(k-1))^\frac{n^2+n}{k}
\end{equation}
Once again, $|B|$ must be square, so we have that either:

1. $\frac{n^2+n}{k}$ is even and $n+1$ is a perfect square

or:

2. $\frac{n^2+n}{k}$ is odd and $n-(k-1)$ is a perfect square.

\end{proof}

We make the note here that our results say nothing about $k$ odd. As one can see from the formula for the determinant of $B$, it must always be square as $\frac{n^2+n}{k}$ is always even for any $n$.

\subsection{Examples}

When $k=4$ Theorem 5.3 shows that there exist no configurations for the values of $n=7,11,16$. Similarly, when $k=6$, there exist no configurations for $n=11, 12, 17$.

That is, there are no configurations with the TOPs property with the following parameters: $(60_8),(136_{12}),(276_{17}),(138_{12}),(162_{13}),(312_{18})$.

\section{$k=3$ in general}

We note now that for $k=3$, any line in the configuration is parallel with two others. In the general case these are not necessarily parallel to each other. However, we can create cycles of lines each parallel to the next in the cycle. This then partitions the lines into cycles where consecutive lines are parallel. 

For example, the configuration denoted $9.1$ of \cite{betten2000counting}, has lines $\{\{0,1,2\},\{3,4,5\},\{6,7,8\},\{0,3,6\},\{0,4,7\},\{1,3,7\},\\\{1,5,8\},\{2,4,8\},\{2,5,6\}\}$. The lines are partitioned into three mutually non-intersecting lines $\{0,1,2\},\{3,4,5\},\{6,7,8\}$ and a cycle of 6 lines $\{0,3,6\},\{0,4,7\},\{1,3,7\},\{1,5,8\},\{2,4,8\},\{2,5,6\}$. We will refer to this configuration as configuration $9.1$.

So now we examine the matrix form $B=N^TN$ and yield the result:

\begin{thm}
\label{thm:2}
    If the lines of a configuration with deficiency $k=3$ are partitioned into $m$ cycles of lines where consecutive lines are non-intersecting, then either the configuration is configuration $9.1$ or $m$ is odd.
\end{thm}

\begin{proof}
So then the incidence matrix $N_i$ for the $i^{th}$ cycle of lines can be expressed such that $N_iN_i^T=B_i$ is a circulant matrix with first row $n+1, 0,1,\ldots, 1, 0$.

From this observation, we note that the matrix $B$ for $k=3$ must be of the form:

$$B=1_{n^2+n+3} + \begin{bmatrix}
A_1 & 0 &\ldots & 0 &0\\
0 & A_2 & &  & \\
\vdots &&\ddots&&\vdots\\
0&&&A_{m-1}&0\\
0&&\ldots &0&A_m\\

\end{bmatrix}$$

Where each $A_i$ is a circulant matrix with with first row $n,-1,0,\ldots,0,-1$ and $1_t$ is notation for a square matrix of dimension $t\times t$ where every element is $1$. 

Let $$A= \begin{bmatrix}
A_1 & 0 &\ldots & 0 &0\\
0 & A_2 & &  & \\
\vdots &&\ddots&&\vdots\\
0&&&A_{m-1}&0\\
0&&\ldots &0&A_m\\

\end{bmatrix}$$

We then apply the matrix determinant lemma \cite{brookes2005matrix} to yield:

$$|B|=|A|(1+e^TA^{-1}e)$$

Where $e$ is a column vector of $n^2+n+3$ ones.

Hence, it remains to compute $|A|$ and $(1+e^TA^{-1}e)$.

Firstly, $|A|=|A_1|\dots|A_m|$ and 

$$A^{-1}= \begin{bmatrix}
A_1^{-1} & 0 &\ldots & 0 &0\\
0 & A_2^{-1} & &  & \\
\vdots &&\ddots&&\vdots\\
0&&&A_{m-1}^{-1}&0\\
0&&\ldots &0&A_m^{-1}\\

\end{bmatrix}$$

So now, if some $A_x$ has dimension $t_x\times t_x,$ we make the observation that $e^TA_xe$ is the sum of all elements of $A_x$ (also known as the grand sum). Therefore, $e^TA_xe=t_x(n-2)$.

We now note that; $(1_{t_x}A_x)_{ij}$ sums the $j^{th}$ column of $A_x$ which is equal to $n-2$. Similarly, $(A_x1_{t_x})_{ij}$ sums the $i^{th}$ row which is also equal to $n-2$. It follows that $1_{t_x}A_x=A_x1_{t_x}$ and we have:

$$e^TA_x^{-1}et_x(n-2)=e^TA_x^{-1}ee^TA_xe=e^TA_x^{-1}1_{t_x}A_xe=e^TA_x^{-1}A_x1_{t_x}e=e^T1_{t_x}e=t_x^2$$

Hence, $$e^TA_x^{-1}e=\frac{t_x}{n-2}$$

So then
$$1+e^TA^{-1}e=1+\sum_{x=1}^m\frac{t_x}{n-2}=1+\frac{n^2+n+3}{n-2}=\frac{n^2+2n+1}{n-2}=\frac{(n+1)^2}{n-2}$$

Now, the eigenvalues of a circulant matrix with first row $c_0,c_1,\ldots,c_{t-1}$ can be calculated as:

$$\sum_{t=0}^{t-1}c_t\omega^{tj}$$

Where $\omega=e^{\frac{2\pi i}{t}}$. \cite{davis1979circulant}

Hence for a given $A_x$ of size $t_x$, the eigenvalues are:

$$n\omega^{0}-\omega^{j}-\omega^{j(t_x-1)}$$

These can be simplified to:

$$n-2cos(\frac{2\pi j}{t_x})$$

Hence the determinant of $A_x$ as the product of the eigenvalues is:

$$\prod_{j=0}^{t_x-1}(n-2cos(\frac{2\pi j}{t_x})$$

For $t_x$ even this can be simplified to:

$$(n-2)(n+2)\prod_{j=1}^{\frac{t_x}{2}-1}(n-2cos(\frac{2\pi j}{t_x}))^2$$

and for $t_x$ odd this can be simplified to:

$$(n-2)\prod_{j=1}^{\frac{t_x-1}{2}}(n-2cos(\frac{2\pi j}{t_x}))^2$$

Putting these together, we yield that the full determinant is:

\begin{equation} \label{eq:2}
(n-2)^{m-1}(n+1)^2(n+2)^{\text{number of even }t_x}\left(\prod_{x\text{ such that }t_x\text{ is even}}\prod_{j=1}^{\frac{t_x}{2}-1}(n-2cos(\frac{2\pi j}{t_x}))^2\right) \left(\prod_{x\text{ such that }t_x\text{ is odd}}\prod_{j=1}^{\frac{t_x-1}{2}}(n-2cos(\frac{2\pi j}{t_x}))^2\right)
\end{equation}

We now observe that $n^2+n+3$ is odd and:

$$n^2+n+3=\sum_{x\text{ such that } t_x\text{ is odd}}t_x+\sum_{x\text{ such that } t_x\text{ is even}}t_x$$

Hence there must be an odd number of odd $t_x$.

If $m$ is even, then there is an odd number of even $t_x$. Hence, for \ref{eq:2} to be square, both $n-2$ and $n+2$ must be square. Only $n=2$ is such a natural number.  Hence the parameters are $(9_3)$. By \cite{betten2000counting}, there are only three such configurations; the Pappus configuration; the cyclic configuration where the lines are $\{s \text{ mod }9,s+1\text{ mod }9,s+3\text{ mod }9\}$ for $s=0,\ldots,8$ denoted configuration $9.2$; and configuration $9.1$ as stated above. The Pappus configuration has the TOPs property and $3$ (odd) cycles where consecutive lines are non-intersecting, configuration $9.2$ has $1$ (odd) cycle where consecutive lines are non-intersecting. It remains then that only configuration $9.1$ has an even ($2$) cycles where consecutive lines are non-intersecting.

If $m$ is odd, then $m-1$ is even. Also as $m$ (odd) is the sum of the number of odd $t_x$ (odd by above) and the number of even $t_x$. Hence the number of even $t_x$ must be even. Hence, for $m$ odd, \ref{eq:2} is always square.

\end{proof}

\bibliographystyle{unsrt}  
\bibliography{references} 

\end{document}